\documentclass{amsart}
\usepackage{amsmath,color,amssymb,tikz,tikz-cd}

\newcommand{\ZZ}{\mathbb{Z}}

\newcommand{\cI}{\mathcal{I}}
\newcommand{\cJ}{\mathcal{J}}
\newcommand{\cH}{\mathcal{H}}
\newcommand{\cK}{\mathcal{K}}
\newcommand{\cM}{\mathcal{M}}

\newcommand{\sgn}{\mathrm{sgn}}
\newcommand{\id}{\mathrm{id}}
\newcommand{\ten}{\otimes}

\newtheorem{theorem}{Theorem}[section]
\newtheorem*{theorem*}{Theorem}
\newtheorem{proposition}[theorem]{Proposition}
\newtheorem{corollary}[theorem]{Corollary}
\newtheorem{lemma}[theorem]{Lemma}
\theoremstyle{definition}
\newtheorem{example}[theorem]{Example}
\theoremstyle{definition}
\newtheorem{definition}[theorem]{Definition}
\newtheorem{remark}[theorem]{Remark}

\title[Cancellation-free antipode formula for matroid Hopf algebra]{A cancellation-free antipode formula (for uniform matroids) for the restriction-contraction matroid Hopf algebra}
\author{Eric Bucher}
\address{Department of Mathematics,
Michigan State University, East Lansing, MI 48824, USA}
\email{ebuche2@math.msu.edu}

\author{Jacob P. Matherne}
\address{Department of Mathematics and Statistics,
University of Massachusetts Amherst, Amherst, MA 01003, USA}
\email{matherne@math.umass.edu}

\thanks{The first author was supported by the LSU VIGRE Grant and the LSU GAANN Grant}
\thanks{The second author was supported by an LSU Dissertation Year Fellowship.}

\begin{document}

\begin{abstract}
In this paper, we give a cancellation-free antipode formula (for uniform matroids) for the restriction-contraction matroid Hopf algebra. The cancellation-free formula expresses the antipode of uniform matroids as a sum over certain ordered set partitions. 
\end{abstract}
\maketitle

\section{Introduction and methods}
\subsection{Introduction}
Many combinatorial structures can be used as building blocks to construct graded connected Hopf algebras. Examples come from graphs, simplices, polynomials, Young tableaux, and many more. In this paper we will focus on the restriction-contraction Hopf algebra of matroids first introduced by Schmitt \cite{s94}.

When a Hopf algebra arises from combinatorial objects, it is often referred to as a \emph{combinatorial Hopf algebra} (see \cite{gr16} for an excellent treatment of Hopf algebras, especially as they appear in combinatorics). Motivations for studying these algebras appear throughout many diverse areas of mathematics, including: combinatorics, representation theory, mathematical physics, and K-theory. Given a combinatorial Hopf algebra, the computation of its antipode gives combinatorial identities for the objects which built up the algebra. For this reason, finding the simplest expression for the antipode can prove extremely valuable. 

In general the antipode is given in terms of satisfying certain commutative diagrams (see Definition \ref{def:hopfalg}), but in the setting where the Hopf algebra is both graded and connected we have a formula given by Takeuchi which describes the map explicitly.

\begin{theorem*}[\cite{t71}]
A graded, connected $\Bbbk$-bialgebra $\cH$ is a Hopf algebra, and it has a unique antipode $S$ whose formula is given by
\begin{equation}\label{eqn:tak}
S = \sum_{i \in \ZZ_{\geq 0}} (-1)^i \mu^{i-1} \circ \pi^{\ten i} \circ \Delta^{i-1}
\end{equation}
where $\mu^{-1} = \eta$, $\Delta^{-1} = \epsilon$, and $\pi: \cH \rightarrow \cH$ is the projection map defined by extending linearly the map
\[
\pi|_{\cH_\ell} = \left\{\begin{array}{ll}0 & \mathrm{if}\ \ell = 0,\\ \id & \mathrm{if}\ \ell \geq 1,\end{array}\right.
\]
where $\cH_{\ell}$ is the $\ell^{\text{th}}$ graded piece of $\cH$.
\end{theorem*}

The only drawback to the formula above is that it can contain a large amount of cancellation. The goal in general is then to find refinements of this formula so that we can more easily compute the antipode of $\mathcal{H}$. Ideally this leads to a cancellation-free formula. Much work has been done in this area for specific Hopf algebras. Humpert and Martin \cite{hm12} found a cancellation-free formula for the incidence Hopf algebra on graphs. Aguiar and Mahajan \cite{am10} introduced an embedding into the Hopf monoid which was utilized by Aguiar and Ardila \cite{aa16} and later by Benedetti and Bergeron \cite{bb16} to compute various cancellation-free formulas for embeddable Hopf algebras. Benedetti and Sagan \cite{bs16} applied a sign-reversing involution to find a cancellation-free antipode formula for the shuffle Hopf algebra, the Hopf algebra of polynomials, and Hopf algebras related to QSym; among others. And very recently Benedetti, Hallam, and Macechek \cite{bhm16} were able to construct a cancellation-free formula for the Hopf algebra of simplicial complexes. 

In this paper, we begin working towards finding a cancellation-free antipode formula for the restriction-contraction Hopf algebra of matroids. We will define the multiplication and the comultiplication for this Hopf algebra in Section \ref{sec:prelims}, as well as give a brief background on matroids. The matroid terminology used throughout the paper will be consistent with that of Oxley's book \cite{o11}. Our main result is a cancellation-free formula for computing the antipode of uniform matroids in this algebra. 

\begin{theorem*}[Main Result]
The image of the uniform matroid $U^m_n$ under the antipode map $S$ is given by the following cancellation-free formula. Choose a total ordering on the ground set of the matroid, $<$:
\[
S(U^m_n) = \displaystyle \sum_{I,L} (-1)^{n-|L|+1} U^{|I|}_I \oplus U^{m-|I|}_L,
\]
 where $I,L$ ranges over all pairs of subsets of the ground set $E$ such that
\begin{itemize}
\item $I$ and $L$ are disjoint, 
\item $|I|<m$, and
\item $|I|+|L|\geq m$, and
\item if $|I|+|L|=m$ then $|L|=1$ and the element in $L$ is the maximal element of $I \cup L$ with respect to $<$.
\end{itemize}
\end{theorem*}

In the future, we hope to utilize similar techniques to extend these results to other classes of matroids.

\subsection{Methods}\label{sec:methodology}

Before we move on to the background work we should briefly talk about the strategy of the proof. The authors believe that similar methodology could be applied to an assortment of Hopf algebras where cancellation-free formulas have still proven elusive. Therefore we would like to give a brief overview of our methods. We utilize similar approaches to those of Benedetti and Sagan \cite{bs16} by applying split-merge involution. 

\begin{enumerate}
\item The first thing that must be done is unpack the Takeuchi formula. We need to understand what this sum looks like in our particular Hopf algebra. In the algebra considered by this paper, this reduces to a sum over certain ordered set partitions (see Proposition \ref{prop:unpacktak}).

\item The Takeuchi summation contains many cancellations. Our next step is to systematically try and pair them off with terms that differ only in the sign. We do this by constructing a sign-reversing involution on our set of ordered set partitions. This involution needs to take one partition to another, and make sure that their associated terms in the Takeuchi summation are identical up to a sign change. 

\item This gives us a way to pair off the terms. Finding a negative term for each positive term by letting the involution guide you to the next counterpart. Therefore the only terms that will remain in the Takeuchi summation after the cancellation will be the terms that correspond to fixed points of the involution. 

\item The next step is to characterize the fixed points of our involution. If these fixed points correspond to unique terms in the Takeuchi summation, we know that this new formula is in fact cancellation free.
\end{enumerate}

This methodology is obviously not unique to antipode computations and can be found throughout combinatorial literature. The authors would like to reiterate that they were following in the footsteps of Benedetti and Sagan when they applied this technique to the restriction-contraction Hopf algebra of matroids.

\subsection*{Acknowledgements}

The authors would like to thank Carolina Benedetti and Bruce Sagan for helpful conversations.  They would also like to thank the LSU Mathematics Department for not evicting them from their offices even after they had graduated.

\section{Preliminaries}\label{sec:prelims}

\subsection{Matroids}

Matroids abstract the notion of linear independence of vectors in a vector space.  We first develop the basics of matroid theory that we will need for the rest of this paper.  We follow the notation of \cite{o11} and point to it for further details on matroids.
\begin{definition}
A \emph{matroid} $M$ is an ordered pair $(E, \cI)$ consisting of a finite set $E$ (called the ground set of $M$) and a collection $\cI$ of subsets of $E$ called independent sets, having the following three properties:
\begin{enumerate}
\item[(I1)] $\emptyset \in \cI$.
\item[(I2)] If $I \in \cI$ and $I^\prime \subseteq I$, then $I^\prime \in \cI$.
\item[(I3)] If $I_1$ and $I_2$ are in $\cI$ and $|I_1| < |I_2|$, then there is an element $e$ of $I_2 - I_1$ such that $I_1 \cup e \in \cI$.
\end{enumerate}
A \emph{loop} $e \in E$ is an element that does not belong to any independent set.
\end{definition}

We say two matroids $M = (E, \cI)$ and $N = (F, \cJ)$ are \emph{isomorphic} if there exists a bijection $f: E \xrightarrow{\sim} F$ between the ground sets which preserves independent sets; that is, if $I \in \cI$, then $f(I) \in \cJ$.  There are several ways to get new matroids from old ones, which we recall below.  For Definitions \ref{def:dirsum}--\ref{def:contract}, let $M = (E,\cI)$ and $N = (F, \cJ)$ be matroids where $E$ and $F$ are disjoint.

\begin{definition}\label{def:dirsum}
The \emph{direct sum of $M$ and $N$}, denoted $M \oplus N$, is the matroid $(E \cup F, \{I \cup J\ |\ I \in \cI, J \in \cJ\})$.
\end{definition}

\begin{definition}
Let $S \subseteq E$.  The \emph{restriction of $M$ to $S$}, denoted $M|_S$, is the matroid $(S, \cK)$ with $\cK = \{I \in \cI\ |\ I \subseteq S\}$.
\end{definition}

\begin{definition}\label{def:restriction}
Let $e \in E$.  The \emph{deletion of $e$ from $M$}, denoted $M \setminus e$, is the matroid $(E - \{e\}, \cI^\prime)$ where $\cI^\prime = \{I \subseteq E - \{e\}\ |\ I \in \cI\}$.
\end{definition}

We do not explicitly make use of deletion in this note:  we only use it in the definition of contraction below.
\begin{definition}\label{def:contract}
Let $e \in E$.  The \emph{contraction of $e$ from $M$}, denoted $M/e$, is the matroid $(E-\{e\},\cI^{\prime \prime})$, where $\cI^{\prime \prime} = \{I \subseteq E - \{e\}\ |\ I \cup \{e\} \in \cI\}$ if $e$ is not a loop, and $\cI^{\prime \prime} = \cI^\prime$ (see Definition \ref{def:restriction}) if $e$ is a loop.

More generally, if $S \subseteq E$, then the \emph{contraction of $S$ from $M$}, denoted $M/S$, is the matroid on $E-S$ gotten by contraction of each element $s \in S$ separately (order does not matter).
\end{definition}

Throughout the paper, when the ground set of a matroid $M = (E,\cI)$ is understood, we sometimes specify $M$ by its independent sets.  In this case, we will write $M = \cI$ as in the next example.

\begin{example}\label{ex:u24}
We give an example of each of the notions in the previous definitions.  Let $E$ be the set $\{v_1,\ldots, v_4\}$ of column vectors of the matrix
\[
\left(\begin{array}{cccc}
1 & 0 & 1 & 1 \\
0 & 1 & 1 & -1
\end{array}\right).
\]
One can check that the set of linearly independent subsets $\cI$ of $\{v_1,\ldots,v_4\}$ make $(E,\cI)$ into a matroid.  In fact,
\[
\cI = \{\emptyset,\{v_1\}, \{v_2\}, \{v_3\}, \{v_4\}, \{v_1,v_2\}, \{v_1,v_3\}, \{v_1,v_4\}, \{v_2,v_3\}, \{v_2,v_4\}, \{v_3,v_4\}\}. 
\]
The restriction $M|_{\{v_2\}} = \{\emptyset, \{v_2\}\}$, the contraction $M/{v_2} = \{\emptyset, \{v_1\}, \{v_3\}, \{v_4\}\}$, and $M|_{\{v_2\}} \oplus M/{v_2} = \{\emptyset, \{v_1\}, \{v_2\}, \{v_3\}, \{v_4\}, \{v_1,v_2\},\{v_2,v_3\},\{v_2,v_4\}\}$.
\end{example}

For the rest of the paper, we may write subsets of matroids without braces and with their elements concatenated (like $v_1v_2$ for $\{v_1,v_2\}$) in hopes of alleviating some notational woes.

We now introduce the class of matroids that we will study in this paper.

\begin{definition}
Let $n,m \in \ZZ_{\geq 0}$ with $m \leq n$.  The \emph{uniform matroid of rank $m$ on an $n$-element set}, denoted $U^m_n$, is the matroid on an $n$-element set $E$ whose collection of independent sets $\cI$ is the set of all subsets of $E$ of cardinality less than or equal to $m$.
\end{definition}

\begin{remark} We will sometimes denote $U^t_I$ to be the uniform matroid $U^t_{|I|}$ with ground set $I$. This allows us to keep track of the ground sets of uniform matroids.
\end{remark}

\begin{definition}
The matroid $U_t^t$ for any $t$ is called the \emph{free matroid on $t$ elements}.  The matroid $U^0_0$ is called the \emph{trivial matroid}.
\end{definition}

\begin{remark}
The matroid in Example \ref{ex:u24} is the uniform matroid $U^2_4$.
\end{remark}

We will make frequent use of the following well-known result whose proof is immediate from the definitions.

\begin{lemma}\label{lem:closedunderminors}
Let $U^m_n$ be a uniform matroid with ground set $E$, let $e \in E$, and let $S \subseteq E$.  We have the following formulas:
\[
U^m_n|_S \cong U^{\min(m,|S|)}_{|S|}
\]
and
\[
U^m_n/e \cong \left\{\begin{array}{ll}U^{m-1}_{n-1} & \text{if } m>0, \\ U^m_{n-1} & \text{if } m = 0. \end{array}\right.
\]
\end{lemma}

We will also make use of the following lemma.

\begin{lemma}\label{lem:free}
If $A$ and $B$ are disjoint finite sets, then the matroid \[U^{|A|}_A \oplus U^{|B|}_B=U^{|A|+|B|}_{A\bigcup B}\]
\end{lemma}

\begin{proof}
We start by utilizing the well-known results found in \cite{o11} that the free matroid $U_t^t$ is the cycle matroid of a forest with $t$ edges and that the direct sum of two cycle matroids is the same as the cycle matroid for the direct sum of the two graphs. Then $U^{|A|}_A$ is the cycle matroid for a forest with $|A|$ edges and $U^{|B|}_B$ is the cycle matroid for a forest with $|B|$ edges. The direct sum of the two graphs is a forest with $|A|+|B|$ edges, and so $U^{|A|}_A \oplus U^{|B|}_B=U^{|A|+|B|}_{A\bigcup B}$.
\end{proof}

\subsection{Hopf algebras}\label{sec:hopfalgs}
For the remainder of this paper, we fix a field $\Bbbk$.  We now work towards the definition of a Hopf algebra.

\begin{definition}\label{def:alg}
An \emph{associative $\Bbbk$-algebra $A$} is a $\Bbbk$-module together with a $\Bbbk$-linear operation $\mu: A \ten A \rightarrow A$ called multiplication, and a $\Bbbk$-linear map $\eta: \Bbbk \rightarrow A$ called a unit.  These maps must make the following diagrams commute
\[
\begin{tikzcd}
A \ten A \ten A \ar{r}{\id \ten \mu} \ar{d}[swap]{\mu \ten \id} & A \ten A \ar{d}{\mu} & & A \ten \Bbbk \ar{d}[swap]{\id \ten \eta} & A \ar{l}[swap]{\iota_1} \ar{d}{\id} \ar{r}{\iota_2} & \Bbbk \ten A \ar{d}{\eta \ten \id} \\
A \ten A \ar{r}[swap]{\mu} & A & & A \ten A \ar{r}[swap]{\mu} & A & A \ten A \ar{l}{\mu}
\end{tikzcd}
\]
where $\iota_1$ and $\iota_2$ are the inclusions.
\end{definition}

\begin{remark}
In the definition of $\Bbbk$-algebra above, the diagram on the left states that $\mu$ is associative, and the diagram on the right asserts that $\eta(1_{\Bbbk}) = 1_A$ (the two-sided multiplicative identity in $A$).
\end{remark}

\begin{definition}\label{def:coassoc}
A \emph{coassociative $\Bbbk$-coalgebra $C$} is a $\Bbbk$-module $C$ together with a $\Bbbk$-linear operation $\Delta: C \rightarrow C \ten C$ called comultiplication, and a $\Bbbk$-linear map $\epsilon: A \rightarrow \Bbbk$ called a counit.  These maps must make the following diagrams (the dual diagrams to those in Definition \ref{def:alg}) commute
\[
\begin{tikzcd}
C \ten C \ten C & C \ten C \ar{l}[swap]{\id \ten \Delta} & & C \ten \Bbbk \ar{r}{p_1} & C & \Bbbk \ten C \ar{l}[swap]{p_2} \\
C \ten C \ar{u}{\Delta \ten \id} & C \ar{l}{\Delta} \ar{u}[swap]{\Delta} & & C \ten C \ar{u}{\id \ten \epsilon} & C \ar{u}{\id} \ar{l}{\Delta} \ar{r}[swap]{\Delta} & C \ten C \ar{u}[swap]{\epsilon \ten \id}
\end{tikzcd}
\]
where $p_1$ and $p_2$ are the projections.
\end{definition}

\begin{definition}
If $A$ is a $\Bbbk$-algebra and a $\Bbbk$-coalgebra, then $A$ is called a \emph{$\Bbbk$-bialgebra} if both $\Delta$ and $\epsilon$ are $\Bbbk$-algebra homomorphisms with respect to $\mu$ and $\eta$.
\end{definition}

\begin{definition}\label{def:hopfalg}
A \emph{Hopf algebra} is a $\Bbbk$-bialgebra $\cH$ together with a $\Bbbk$-linear map $S: \cH \rightarrow \cH$ (called an \emph{antipode}) such that $\mu \circ (S \ten \id) \circ \Delta = \mu \circ (\id \ten S) \circ \Delta = \eta \circ \epsilon$; that is, such that $S$ makes the following diagram commute:
\[
\begin{tikzcd}
\ & \cH \ten \cH \ar{rr}{S \ten \id} & & \cH \ten \cH \ar{dr}{\mu} & \\
\cH \ar{ur}{\Delta} \ar{rr}{\epsilon} \ar{dr}[swap]{\Delta} & & \Bbbk \ar{rr}{\eta} & & \cH \\
\ & \cH \ten \cH \ar{rr}[swap]{\id \ten S} & & \cH \ten \cH \ar{ur}[swap]{\mu} & 
\end{tikzcd}
\]
\end{definition}

Certainly not every $\Bbbk$-bialgebra $A$ is a Hopf algebra.  However, if $A$ is graded and connected, then there exists a unique antipode $S$ making it into a Hopf algebra.

\begin{definition}
A $\Bbbk$-bialgebra is called \emph{graded} if $\displaystyle A = \bigoplus_{i \in \ZZ_{\geq 0}} A_i$ and each of $\mu, \eta, \Delta,$ and $\epsilon$ are graded $\Bbbk$-linear maps.  If $A_0 = \Bbbk$, then $A$ is said to be \emph{connected}.
\end{definition}

If $\cH$ is a graded, connected bialgebra, then Takeuchi showed that $\cH$ is a Hopf algebra and gave a formula for the antipode of any element \cite{t71}.

\begin{theorem}[\cite{t71}]
A graded, connected $\Bbbk$-bialgebra $\cH$ is a Hopf algebra, and it has a unique antipode $S$ whose formula is given by
\begin{equation}\label{eqn:tak}
S = \sum_{i \in \ZZ_{\geq 0}} (-1)^i \mu^{i-1} \circ \pi^{\ten i} \circ \Delta^{i-1}
\end{equation}
where $\mu^{-1} = \eta$, $\Delta^{-1} = \epsilon$, and $\pi: \cH \rightarrow \cH$ is the projection map defined by extending linearly the map
\[
\pi|_{\cH_\ell} = \left\{\begin{array}{ll}0 & \mathrm{if}\ \ell = 0,\\ \id & \mathrm{if}\ \ell \geq 1.\end{array}\right.
\]
\end{theorem}

As stated in the introduction, this formula for $S$ often has a great deal of cancellation in examples.  In this paper, we are concerned with the restriction-contraction Hopf algebra of matroids, which we introduce in the next section.  For this Hopf algebra, we use the split-merge technique introduced in \cite{bs16} to reinterpret (\ref{eqn:tak}) as one with no cancellation.

\subsection{The restriction-contraction matroid Hopf algebra}\label{sec:resconmathopfalg}
We define the main object of study of this note, the restriction-contraction matroid Hopf algebra, which was originally defined in \cite{s94}.  

\begin{definition}
A \emph{minor of a matroid $M$} is any matroid that is gotten by performing a sequence of restrictions and contractions on $M$.
\end{definition}

Let $\cM$ be a collection of matroids closed under taking minors, and let $\widetilde{\cM}$ be the set of isomorphism classes of matroids in $\cM$.  Then direct sum endows $\widetilde{\cM}$ with an associative product \cite{s94} so that we can form the monoid algebra $\Bbbk \widetilde{\cM}$.

\begin{proposition}
Let $\cM$ be a collection of matroids closed under direct sums and minors, and let $M = (E, \cI), N = (F, \cJ) \in \cM$.  Let $\Bbbk \widetilde{\cM}$ be the monoid algebra above with unit map $\eta: \Bbbk \rightarrow \Bbbk \widetilde{\cM}$ (so that $\eta(1_{\Bbbk}) = U^0_0$).  Then $\Bbbk \widetilde{\cM}$ is a Hopf algebra with the following maps:
\[\begin{array}{rcl}
\Bbbk \widetilde{\cM} \ten \Bbbk \widetilde{\cM} & \stackrel{\mu}{\longrightarrow} & \Bbbk \widetilde{\cM} \\
(M,N) & \longmapsto & M \oplus N, \\
& & \\
\Bbbk \widetilde{\cM} & \stackrel{\Delta}{\longrightarrow} & \Bbbk \widetilde{\cM} \ten \Bbbk \widetilde{\cM} \\
M & \longmapsto & \displaystyle \sum_{A \subseteq E} M|_A \ten M/A, \\
& & \\
\Bbbk \widetilde{\cM} & \stackrel{\epsilon}{\longrightarrow} & \Bbbk \\
M & \longmapsto & \left\{\begin{array}{ll}1_{\Bbbk} & \mathrm{if}\ E = \emptyset,\\ 0 & \mathrm{else}.\end{array}\right. \\
\end{array}\]
\end{proposition}

\begin{remark}
We need to choose a suitable class of matroids $\cM$.  Since, in this paper, we are only concerned with computing the antipode for uniform matroids, we could choose $\cM$ to be the set of all uniform matroids.  However, this collection is closed under minors (see Lemma \ref{lem:closedunderminors}) but not under direct sums.  Instead we take $\cM$ to be any suitable collection of matroids consisting of direct sums of minors of uniform matroids.  For example, the collection of all partition matroids (those matroids which are direct sums of uniform matroids) would suffice.
\end{remark}

\section{A cancellation-free antipode formula}
Given a finite set $E$, an \emph{ordered set partition} is a sequence of nonempty disjoint subsets $\pi = (B_1, B_2, \ldots, B_k)$ such that $\cup_i B_i = E$.  The $B_i$ are called \emph{parts}.  If $\pi$ is an ordered set partition of $E$, then we write $\pi \vDash E$.  Sometimes, we will relax the condition that parts must be nonempty.  To denote that some of the $B_i$ may be empty, we write $(B_1, B_2, \ldots, B_k) \vDash_0 E$.  Given a matroid $M = (E, \cI) \in \cM$, let $\Pi_M$ denote the set of all ordered set partitions of the ground set $E$.  

\begin{proposition}\label{prop:unpacktak}
Let $M = (E, \cI) \in \cM$.  Then for $\Bbbk \widetilde{\cM}$, Takeuchi's formula (\ref{eqn:tak}) takes the form:
\[
S(M) = \sum_{k \geq 0} (-1)^k \sum_{(B_1,\ldots,B_k) \vDash E} M|_{B_1} \oplus (M/B_1)|_{B_2} \oplus \cdots \oplus (M/\bigcup_{i=1}^{k-2}B_i)|_{B_{k-1}} \oplus M/\bigcup_{i=1}^{k-1}B_i.
\]
\end{proposition}

\begin{proof}
First note that
\[
\Delta(M) = \sum_{A \subseteq E} M|A \ten M/A = \sum_{(B_1,B_2)\vDash_0 E} M|B_1 \ten M/B_1.
\]
Since $\Delta$ is coassociative (see Definition \ref{def:coassoc}), it follows that
\[
\Delta^{k-1}(M) = \sum_{(B_1, \ldots, B_k) \vDash_0 E} M|_{B_1} \ten (M/B_1)|_{B_2} \ten \cdots \ten (M/\bigcup_{i=1}^{k-2}B_i)|_{B_{k-1}} \ten M/\bigcup_{i=1}^{k-1}B_i.
\]
Since $\pi$ kills $\Bbbk \widetilde{\cM}_0$, substituting this into (\ref{eqn:tak}) yields
\[
S(M) = \sum_{k \geq 0} (-1)^k \sum_{(B_1,\ldots,B_k) \vDash E} M|_{B_1} \oplus (M/B_1)|_{B_2} \oplus \cdots \oplus (M/\bigcup_{i=1}^{k-2}B_i)|_{B_{k-1}} \oplus M/\bigcup_{i=1}^{k-1}B_i,
\]
as desired.
\end{proof}

In order to produce a cancellation-free antipode formula for the restriction-contraction matroid Hopf algebra $\Bbbk \widetilde{\cM}$, we will utilize the split-merge strategy introduced in \cite{bs16}.  To use this method, we proceed as outlined in Section \ref{sec:methodology}.  First we will show some examples of the computation of the Takeuchi formula for a given ordered set partition $\pi$. That is, given an ordered set partition $\pi$, we show how to construct the associated signless term $T(\pi)$ in the antipode formula. Then we will prove this in generality.

%

\subsection{Computing terms in Takeuchi's formula}
Our running example will be the uniform matroid $M=U^3_5$, which has ground set $E = \{a,b,c,d,e\}$ and independent sets
\[
\cI = \{\emptyset, a, b, c, d, e, ab, ac, ad, ae, bc, bd, be, cd, ce, de,\] \[abc, abd, abe, acd, ace, ade, bcd, bce, bde, cde\}.
\]

\begin{example}
Here are a few examples of computing $T(\pi)$.
\begin{itemize} 
\item Let $\pi_1=ab,cd,e$. Then $T(\pi_1)$ in $S(U^3_5)$ is given by
\[
M|_{ab} \oplus (M/a/b)|_{cd} \oplus M/a/b/{cd} = \{\emptyset,a,b,ab\} \oplus \{\emptyset,c,d\} \oplus \{\emptyset\} \simeq U^2_2 \oplus U^1_2.
\]

\item Let $\pi_2=a,bcd,e$. Then $T(\pi_2)$ in $S(U^3_5)$ is given by
\[
M|_{a} \oplus (M/a)|_{bcd} \oplus M/a/{bcd} = \{\emptyset, a\} \oplus \{\emptyset, b, c, d, bc, bd, cd\} \oplus \{\emptyset\} \simeq U^1_1 \oplus U^2_3.
\]

\item Let $\pi_3=a,bcde$. Then $T(\pi_3)$ in $S(U^3_5)$ is given by
\[
M|_a \oplus M/a = \{\emptyset,a\} \oplus \{\emptyset,b,c,d,e,bc,bd,be,cd,ce,de\} \simeq U^1_1 \oplus U^2_4.
\]
\end{itemize}
Notice that the superscripts sum to three and the subscripts are associated to the sizes of the parts of the partition. Below we make this precise for general $\pi$.
\end{example}

\begin{theorem}\label{thm:computetakterms}
Let $M=U_n^m$ with ground set $\{e_j^t\}$. Given an ordered set partition of the ground set $\pi=e^1_1 \cdots e^{k_1}_1, e^1_2 \cdots e^{k_2}_2, \ldots, e^1_j \cdots e^{k_j}_j$, define  $\ell$ to be the first integer such that $k_1 + k_2 + \cdots + k_{\ell} \geq m$. Then $T(\pi)$ in $S(M)$ can be computed as follows:

\begin{itemize}
\item If $k_1 + k_2 + \cdots + k_{\ell} = m$, then $T(\pi)$ in $S(U^m_n)$ is $U^m_m$ with ground set $\displaystyle \bigcup_{i=1}^{\ell} \{e_i^1,\ldots, e_i^{k_i}\}$. 

\item If $k_1 + k_2 + \cdots + k_{\ell} > m$, then $T(\pi)$ in $S(U^m_n)$ is 
\[
U^{k_1 + \cdots + k_{\ell-1}}_{k_1 + \cdots + k_{\ell-1}} \oplus U^{m-k_1 - k_2 - \cdots - k_{\ell - 1}}_{k_{\ell}}
\]
where the first summand has ground set $\displaystyle \bigcup_{i = 1}^{\ell - 1} \{e^1_i, \ldots, e^{k_i}_i\}$ and the second summand has ground set $\{e_{\ell}^1, \ldots, e_{\ell}^{k_{\ell}}\}$.
\end{itemize}
\end{theorem}

\begin{proof}
Since comultiplication is associative we can apply it in any order to a direct sum.  Therefore let us first look at restriction/contraction with respect to the set $I=\bigcup_{i = 1}^{\ell - 1} \{e^1_i, \ldots, e^{k_i}_i\}$. The resulting matroid is
\[
M|_I \oplus M/I. 
\]
Since $k_1 + k_2 + \cdots + k_{\ell-1} < m$, we know that $M|_I$ is the free matroid on the set $I$. Then $(M|_I)|_S$ is the free matroid on the set $S$ for any $S \subset I$. Also $(M|_I)/S$ is the free matroid on the set $I-S$ for any $S \subset I$. Therefore by repeatedly applying Lemma \ref{lem:closedunderminors}, we see that if $I_1, \dots,I_{\ell -1}$ are the first $\ell -1$ parts in $\pi$ we know that after performing all the comultiplication regarding these parts the resulting matroid is 
\[
M|_{I_1} \oplus \cdots \oplus M|_{I_{\ell-1}} \oplus M/I. 
\]

Note that the direct sum of free matroids is itself a free matroid on the union of the ground sets, by Lemma \ref{lem:free}. Therefore  
\[
M|_{I_1} \oplus \cdots \oplus M|_{I_{\ell-1}} \oplus M/I \simeq U_I^{|I|} \oplus M/I.
\]
Now if we apply comultiplication (restrict/contract) to the $\ell^{th}$ part, we will only be changing the summand $M/I$. Again, repeated application of Lemma \ref{lem:closedunderminors} results in the following matroid:

\[
U_I^{|I|} \oplus U_L^{m-|I|} \oplus M/I/L. 
\]

Computing $M/I/L$ requires us to contract $|I|+|L|$ elements and $|I|+|L|\geq m$. The largest independent set in $M$ has cardinality $m$, and therefore $M/I/L$ is the trivial matroid. Therefore,
\[
U_I^{|I|} \oplus U_L^{m-|I|} \oplus M/I/L \simeq U_I^{|I|} \oplus U_L^{m-|I|}.
\]

Notice that in the case where $k_1 + k_2 + \cdots + k_{\ell} = m$ we have 

\[
 U_I^{|I|} \oplus U_L^{m-|I|} \simeq U_{I \cup L}^m.
\] 
\end{proof}

With this construction we get the following corollary.

\begin{corollary}\label{cor:taktermdetbypair}
Given an ordered partition $\pi$, its associated signless term $T(\pi)$ in the antipode formula is determined by a pair of sets $(I, L)$, where
\begin{itemize}
\item $I$ is the set of ground set elements which appear in the first $\ell - 1$ parts of the ordered partition, and
\item $L$ is the $\ell^{\text{th}}$ part.
\end{itemize}
\end{corollary}

\begin{remark}\label{remark}
We want to quickly discuss a special case. Let partitions $\pi_1$ and $\pi_2$ have pairs $(I_1,L_1)$ and $(I_2,L_2)$ respectively, such that both $|I_1| + |L_1|=m$, $|I_2| + |L_2|=m$, and $I_1 \cup L_1 = I_2 \cup L_2$. Then notice that $T(\pi_1)=T(\pi_2)$. This is not an equality that generally holds, but only in the case where $|I| +|L| =m$. Throughout the remainder of the paper we will see this show up as a special case in many of our formulas and algorithms, due to these additional equivalencies.

\end{remark}

\subsection{Defining $\iota_{<}$}
In this section, we will define the sign-reversing involution $\iota_<$. We get a sign-reversing involution for every total ordering, $<$, of the ground set. The result is a map $\iota_<:\Pi_M\rightarrow \Pi_M$. After defining $\iota_<$, we will provide some computational examples. 

Start by fixing a total ordering on the ground set $E$, call it $<$.  Let $\pi=e^1_1 \cdots e^{k_1}_1, e^1_2 \cdots e^{k_2}_2, \ldots, e^1_j \cdots e^{k_j}_j$ be an ordered partition of $E$.  The involution $\iota_<$ will be applied to the parts of $\pi$ starting with the first part. It will first attempt to \textbf{split} the part. If successful no other parts of $\pi$ are changed. If it cannot split the part, then $\iota_<$ will attempt to \textbf{merge} the part with the part immediately following it. If successful no other parts of $\pi$ are changed. If it can neither split nor merge the part, then $\iota_<$ will move on to the next part working left to right. If $\iota_<$ is not applicable to any part of $\pi$, then $\pi$ is a fixed point of $\iota_<$.

Now we will explain what we mean by splitting and merging.
\begin{description}
\item[Splitting] If  the part $e^1_i \cdots e^{k_i}_i$ has more than one element, $\iota$ will try to split this part.  
\begin{itemize}
\item If $k_1 + k_2 + \cdots + k_{\ell} = m$ or $i \neq \ell$, then the part splits.  Let $e^{\text{max}}_i$ be the largest element in $e^1_i \cdots e^{k_i}_i$ with respect to $<$. Now, replace the part $e^1_i \cdots e^{k_i}_i$ with two parts
\[
e_i^{\text{max}}\ \ \ \ \ \ \ \text{and}\ \ \ \ \ \ \ e^1_i \cdots \widehat{e^{\text{max}}_i} \cdots e^{k_i}_i,
\]
where $\widehat{e^{\text{max}}_i}$ means that we omit this element.
\item If $k_1 + k_2 + \cdots + k_{\ell} > m$ and the part is the $\ell^{th}$ part ($i = \ell$), then the involution $\iota_<$ does not apply to this part.
\end{itemize}

\item[Merging] If the part is a single element $e_i^1$, we will try to merge it with the part immediately following it.

\begin{itemize}
\item If $i \neq \ell$ and $e^1_i > e^t_{i+1}$ for all $1 \leq t \leq k_{i+1}$, then it merges.  In this case, we replace $e^1_i,e^1_{i+1}\cdots e^{k_{i+1}}_{i+1}$ with a single part $e^1_i e^1_{i+1} \cdots e^{k_{i+1}}_{i+1}$.
\end{itemize}
\end{description}

\begin{example}
Let us do an example of splitting. Consider the uniform matroid $U^3_5$ with ground set $E = \{a,b,c,d,e\}$ with the total ordering $a<b<c<d<e$.  Consider the ordered partition $\pi=ab, c, de$.  Thus, $\ell = 2$. We check the parts from left to right. Therefore we first check if $\iota$ applies to the first part. Since the first part has more than one element, we check if splitting applies to that part. In this case it does, since $i=1\neq\ell$. We can see that $e^{\text{max}}_i = b$ and therefore
\[
\iota_<(ab, c, de) = b,a,c,de.
\]
\end{example}

\begin{example}
For an example of merging, let us again consider $U^3_5$ with ground set $E = \{a,b,c,d,e\}$ with the total ordering $a<b<c<d<e$. For this example, we will consider the ordered partition $a,c,b,de$.  Here, $\ell = 3$.  We attempt to merge the first part $(i = 1)$.  We cannot since $a < c$.  Then we attempt to merge the second part $(i=2)$.  In this case, we can merge because $c > b$.  Therefore $\iota_<$ applies to the second part and
\[
\iota(a,c,b,de) = a,bc,de.
\]
\end{example}

\begin{theorem}
The map $\iota_<:\Pi_M\rightarrow \Pi_M$ is an involution.
\end{theorem}

\begin{proof}
We will apply $\iota$ twice to an arbitrary ordered partition 
\[
\pi = e^1_1 \cdots e^{k_1}_1, e^1_2 \cdots e^{k_2}_2, \ldots, e^1_j \cdots e^{k_j}_j \in \Pi_M.
\]
There are several cases.
\begin{description}
\item[Case One] If $\pi$ is a fixed point of $\iota_<$, then $\iota_<^2(\pi) = \iota_<(\iota_<(\pi)) = \iota_<(\pi) = \pi$.
\item[Case Two] If $\pi$ is split by $\iota_<$, then
\[
\iota_<(\pi) = e^1_1 \cdots e^{k_1}_1,\ldots,e^{\text{max}}_i,e^1_i\cdots \widehat{e^{\text{max}}_i}\cdots e^{k_i}_i,\ldots,e^1_j\cdots e^{k_j}_j
\]
for some $i$ where $\iota_<$ did not apply to the parts $1$ through $i-1$.  Then $\iota_<$ will not apply to the first $i-1$ parts of $\iota_<(\pi)$.  It will first attempt to merge $e^{\text{max}}_i$ with $e^1_i \cdots \widehat{e^{\text{max}}_i}\cdots e^{k_i}_i$.  It will succeed, since by definition $e^{\text{max}}_i > e^t_i$ for all $1 \leq t \leq k_i$. Therefore, $\iota_<^2(\pi) = \pi$.
\item[Case Three] If $\pi$ is merged by $\iota_<$, then this argument is very similar to that of case two and is left to the reader to verify.
\end{description}
\end{proof}

\subsection{Characterizing the fixed points of $\iota_<$}
Now we will compute the fixed points of $\iota_<$.

\begin{theorem}\label{thm:fixed}
There are two varieties of fixed points of $\iota_<$. They are ordered set partitions with the following criteria:

\begin{itemize} 
\item $k_i = 1$ for all $i\neq \ell$ (note that $k_\ell$ may or may not also be $1$), and
\item $e_1 < e_2 < e_3 < \ldots < e_{\ell - 1}$ and $e_{\ell + 1} < e_{\ell + 2} < \ldots < e_j$.
 
\item Additionaly if $k_1 + k_2 + \cdots + k_{\ell} = m$, then $k_{\ell}=1$ and $e_{\ell-1}<e_{\ell}$.
\end{itemize}

Here we have suppressed the superscript because each part consists of a single element.
 \end{theorem}

\begin{proof}
Let $\pi = e^1_1 \cdots e^{k_1}_1, e^1_2 \cdots e^{k_2}_2, \ldots, e^1_j \cdots e^{k_j}_j$ be a fixed point of $\iota_<$. No part of $\pi$ splits under $\iota_<$. The $i^{th}$ part does not split if and only if $k_i=1$ for all $i \neq \ell$. If  $k_1 + k_2 + \cdots + k_{\ell} > m$ then the $\ell^{th}$ part cannot split either. If  $k_1 + k_2 + \cdots + k_{\ell} = m$ then the $\ell^{th}$ part cannot split if and only if $k_{\ell}=1$.

Now that we know that all the parts are size one except possibly the $\ell^{th}$ part, we will suppress the superscript in the partition notation and simply refer to the $i^{th}$ part as $e_i$.

If $k_{\ell} =1$, no part of $\pi$ merges under $\iota_<$ if and only if $e_i < e_{i+1}$ for all $1 \leq i \leq \ell-1$ and for $\ell+1 \leq i \leq j$. In other words the ground set elements that occur in the first $\ell$ parts must occur in increasing order with respect to $<$. Similarly, so must the elements of the ground set that occur after the $\ell^{th}$ part.

If $k_{\ell} \neq 1$, no part of $\pi$ merges under $\iota_<$ if and only if $e_i < e_{i+1}$ for all $1 \leq i \leq \ell-2$ and for $\ell+1 \leq i \leq j$. In other words the ground set elements that occur in the first $\ell-1$ parts must occur in increasing order with respect to $<$. Similarly, so must the elements of the ground set that occur after the $\ell^{th}$ part.
\end{proof}

\begin{corollary}\label{cor:UL}
The fixed points of $\iota_<$ are in one-to-one correspondence with pairs $(I,L)$ of subsets of the ground set $E$, where

\begin{itemize}
\item $I$ and $L$ are disjoint, 
\item $|I|<m$, and
\item $|I|+|L|\geq m$, and
\item if $|I|+|L|=m$ then $|L|=1$ and $e_i < e_\ell$ for all $e_i \in I$.
\end{itemize}
\end{corollary}

\begin{proof}
Let $\pi$ be a fixed point of $\iota_<$. Then $\pi$ is determined by the ground set elements that show up before the $\ell^{th}$ part, call them $I$, the ground set elements in the $\ell^{th}$ part, call them $L$, and the ground set elements that show up after the $\ell^{th}$ part, call them $J$. Any two of these sets $I,L,J$ will determine the third set. The restrictions in the above corollary are exactly the restrictions needed to make sure that $\pi$ is a proper partition, that $L$ is actually the $\ell^{th}$ part of $\pi$, and that the partition satisfies the criteria of Theorem~\ref{thm:fixed}.
\end{proof}

\begin{example}
Again we consider the uniform matroid $U^3_5$ with the ordered ground set $a < b < c < d < e$.  Below are some (but not all) examples of $\iota_<$-fixed points:
\begin{itemize}
\item $a,b,cd,e$
\item $a,b,c,d,e$
\item $a,c,de,b$
\item $b,c,de,a$
\item$a,b,d,c,e$
\end{itemize}
\end{example}

\begin{definition}
Define the sign of an ordered set partition $\pi$ as 
\[
\sgn(\pi) = \left\{\begin{array}{ll}- & \text{if } \pi \text{ has an odd number of parts,}\\+ & \text{if } \pi \text{ has an even number of parts}.\end{array}\right.
\]

Note that $\sgn(\pi)$ determines whether the term that $\pi$ contributes to Takeuchi's summation is $T(\pi)$ or $-T(\pi)$.
\end{definition}

\begin{lemma}\label{lem:sgn}
Let $\pi$ be an ordered set partition. Then $\sgn(\iota_<(\pi))=-\sgn(\pi)$ if $\pi$ is not a fixed point of $\iota_<$.
\end{lemma}

\begin{proof}
Suppose $\pi$ is not a fixed point.  Then we know $\sgn(\iota_<(\pi)) = -\sgn(\pi)$ because $\iota_<(\pi)$ has exactly one more or one fewer part than $\pi$.
\end{proof}

\begin{lemma}\label{lem:inv}
Let $\pi$ be an ordered set partition. Then $T(\iota_<(\pi))=T(\pi)$ in $S(U_n^m)$.
\end{lemma}

\begin{proof}
Let $\pi$ be any ordered set partition of the ground set of $U_n^m$. Now recall that $T(\pi)$ in $S(U^m_n)$ is determined by the pair $(I,L)$ by Corollary \ref{cor:taktermdetbypair}.  Since $\iota_<$ does not apply to the $\ell^{th}$ part if $k_1 + k_2 + \cdots + k_{\ell} \neq m$, the pair corresponding to $\iota_<(\pi)$ is the same as that of $\pi$.  Hence,  $T(\iota_<(\pi))=T(\pi)$. If $k_1 + k_2 + \cdots + k_{\ell} = m$ then let $\pi'=\iota_<(\pi)$. Then $\pi'$ also has the property that $k_1' + k_2' + \cdots + k_{\ell}' = m$ because we either move an element in $L$ into $I$ or from $I$ into $L$. Hence from Remark \ref{remark} $T(\pi)=T(\pi')$.
\end{proof}

\begin{theorem}\label{thm:main}
We have the following formula for the antipode of uniform matroids. Let $E$ be the ground set for the matroid $U_n^m$. Then,
\[
S(U^m_n) = \displaystyle \sum_{\pi} \sgn(\pi) U^{|I|}_I \oplus U^{m-|I|}_L,
\]
where $\pi$ ranges over all ordered set partitions of $E$ that are fixed points of $\iota_<$.
\end{theorem}

\begin{proof}
By Lemma \ref{lem:sgn} and Lemma \ref{lem:inv} we see that $\iota_<$ is a sign-reversing involution on $\Pi_{U^m_n}$, the set of ordered set partitions of E.  It follows that we can sum over the fixed points of $\iota_<$. Therefore 

\begin{align*}
S(U^m_n) &= \displaystyle \sum_{\text{fixed } \pi} \sgn(\pi) T(\pi)\\
&= \displaystyle \sum_{\text{fixed } \pi} \sgn(\pi) U^{|I|}_I \oplus U^{m-|I|}_L. \hspace{2cm} (\text{By Theorem \ref{thm:computetakterms}})
\end{align*}
\end{proof}

\begin{corollary}
The summation in Theorem \ref{thm:main} is cancellation free and can also be expressed in terms of $(I,L)$.
\[
S(U^m_n) = \displaystyle \sum_{I,L} (-1)^{n-|L|+1} U^{|I|}_I \oplus U^{m-|I|}_L,
\]
 where $I,L$ ranges over all pairs of subsets of $E$ such that
\begin{itemize}
\item $I$ and $L$ are disjoint, 
\item $|I|<m$, and
\item $|I|+|L|\geq m$, and
\item if $|I|+|L|=m$ then $|L|=1$ and $e_i < e_\ell$ for all $e_i \in I$.
\end{itemize}
\end{corollary}

\begin{proof}
To obtain the formula combine Corollary \ref{cor:UL} and Theorem \ref{thm:main}. 

To see that this is in fact cancellation free we need to show that no two distinct pairs $(I_1,L_1)$ and $(I_2,L_2)$ which satisfy the conditions have the same term in the summation. We will prove this by contradiction. Consider their terms $U_{I_1}^{|I_1|} \oplus U_{L_1}^{m-|I_1|}$ and $U_{I_2}^{|I_2|} \oplus U_{L_2}^{m-|I_2|}$ and assume that $U_{I_1}^{|I_1|} \oplus U_{L_1}^{m-|I_1|} = U_{I_2}^{|I_2|} \oplus U_{L_2}^{m-|I_2|}$. Then $I_1 \cup L_1=I_2 \cup L_2$. Additionally either $I_1=I_2$ and $L_1=L_2$, or $|I_1|+|L_1|=|I_2|+|L_2|=m$ \cite{o11}. Since they are distinct terms, we know the first situation cannot occur. Therefore  $|I_1|+|L_1|=|I_2|+|L_2|=m$ and so $|L_1|=|L_2|=1$. Since $e_i <e_{\ell}$ for each element $e_i \in I$, we know that there is exactly one choice $L_t$ and so the two pairs are the same. This is a contradiction---thus, there was only ever one pair satisfying our condition that produces the term  $U^{|I|}_I \oplus U^{m-|I|}_L$. Hence, our summation is cancellation free.
\end{proof}

\begin{remark}
Notice that despite its apparent dependence on the choice of total ordering, the antipode is actually independent of the ordering. Each ordering will give a distinct involution which can be used to produce the antipode.
\end{remark}

\section{Future Plans}
In the case of the restriction-contraction Hopf algebra for matroids, the authors believe that the above techniques can be generalized to give a cancellation-free formula for the antipode of more general matroids. They have given some consideration towards computing the antipode for direct sums of matroids whose antipodes are known. Since a decomposition formula for generic matroids is currently not known, this is not sufficient to produce an antipode formula for all matroids.  In any case, the next step will be to attempt to utilize the techniques of this paper to find cancellation-free formulas for other families of matroids.  We hope to add these results to this document in a future version.

\bibliography{refs}{}
\bibliographystyle{alpha}

\end{document}